\documentclass[12pt, a4paper]{article}
\pdfoutput=1
\usepackage{lspaper}

\renewcommand\sss[1][n]{\mathfrak{S}_{#1}}
\renewcommand\spe[1]{\mathcal{S}^{#1}}
\newcommand\shp[1]{\operatorname{Shp}(#1)}

\begin{document}

\title{On the semisimplicity of the cyclotomic quiver Hecke algebra of type $C$}

\runninghead{On the semisimplicity of the cyclotomic quiver Hecke algebra of type $C$}


\toptitle

\begin{abstract}
We provide criteria for the cyclotomic quiver Hecke algebras of type $C$ to be semisimple.
In the semisimple case, we construct the irreducible modules.
\end{abstract}

\section{Introduction}

The \emph{quiver Hecke algebras} $\scrr_n$ were introduced by Khovanov and Lauda~\cite{kl09} and Rouquier~\cite{rouq} to categorify the negative half of quantum groups.
Kang and Kashiwara~\cite{kk12} later showed that cyclotomic quotients $\scrr^\La_n$ of $\scrr_n$ categorify irreducible highest weight modules with dominant integral highest weight $\La$.
Motivated and propelled by an isomorphism theorem of Brundan and Kleshchev \cite{bkisom}, these cyclotomic quotients have received a lot of attention in types $A_\infty$ and $A^{(1)}_\ell$.
However, in other types relatively little is known about the cyclotomic quiver Hecke algebras.
Among the few results here are Ariki and Park's results on the representation type of their blocks when $\La = \La_0$~\cite{apa2,apc,apd}.

One of the first questions one should ask when studying a finite-dimensional algebra is whether or not it is semisimple.
In this short note, we will prove semisimplicity criteria for the cyclotomic quiver Hecke algebras $\scrr^\La_n$ in type $C$, over a field, building on previous work \cite{aps}, in which we developed a Specht module theory in types $C_\infty$ and $C^{(1)}_\ell$.
Our result is a fundamental step in gaining a better understanding of these algebras.

Now we state our main result -- see \cref{sec: background} for definitions of the notation used.

\begin{thm}[Main Theorem]\label{main}
Over a field, $\scrr^\La_n$ is semisimple if and only if the following two conditions are satisfied.
\begin{enumerate}[label=(\roman*)]
\item For all $i \in I$, $\langle\La,\alpha^\vee_{i,n}\rangle \leq 1$.
\item For all $1\leq j\leq l$, $\tfrac{n-1}{2} \leq \overline{\kappa_j} \leq \ell - \tfrac{n-1}{2}$.
\end{enumerate}
\end{thm}

Our proof that $\scrr^\La_n$ is semisimple when the above two conditions hold is inspired by an argument from Mathas's survey~\cite{m14surv} in type $A$.
In the other direction, when the conditions fail, we explicitly construct modules that have one-dimensional submodules, which we show have no complement, thus concluding that $\scrr^\La_n$ is not semisimple.
In most cases, the modules we construct are in fact Specht modules, and our previous work with Ariki and Park~\cite{aps} is crucial to our proof.

\begin{ack}
The author is an International Research Fellow of the Japan Society for the Promotion of Science.
We thank Professor Susumu Ariki for useful discussions on the contents of this paper, and Professor Andrew Mathas for helpful comments on an earlier version.
We also thank JSPS for their generous financial support.
Finally, we thank the referee for their helpful comments.
\end{ack}

\section{Background}\label{sec: background}

We begin by providing a brief summary of the necessary definitions.
Throughout, we let $\calo$ denote an arbitrary integral domain.
All our modules are left modules.

\subsection{The quiver Hecke algebras}

Let $\ell \in \{2,3,\dots\} \cup \{\infty\}$, and set $I:=\bbz/(\ell+1)\bbz$ if $\ell<\infty$ or $I = \bbz_{\geq0}$ if $\ell = \infty$.
If $\ell<\infty$, we identify $I$ with the set $\{0,1,2,\dots,\ell\}$.
We adopt standard notation from \cite{Kac} for the root datum of type $C^{(1)}_\ell$ or $C_\infty$.
In particular, we have \emph{simple roots} $\{\alpha_i \mid i\in I\}$, \emph{simple coroots} $\{\alpha^\vee_i \mid i \in I\}$, and we have \emph{fundamental weights} $\{\La_i \mid i \in I\}$ in the \emph{weight lattice} $\sfp$.
We let $\sfq^+:= \bigoplus_{i\in I} \bbz_{\geq 0} \alpha_i$ be the \emph{positive cone of the root lattice} and $\sfp^+:= \{\La\in \sfp\mid \langle \La, \alpha^\vee_i \rangle \geq 0 \text{ for all } i\in I\}$ the \emph{positive weight lattice}, where $\langle - , -\rangle$ is the natural pairing $\langle \La_i, \alpha^\vee_j \rangle = \delta_{i,j}$.
We say that $\beta = \sum_{i\in I} a_i \alpha_i \in \sfq^+$ has \emph{height} $\operatorname{ht}(\beta) = \sum_{i\in I} a_i$, and $\La = \sum_{i\in I} b_i \La_i \in \sfp^+$ has \emph{level} $\sum_{i\in I} b_i$.
Set $\sfq^+_n:=\{\beta \in \sfq^+ \mid \operatorname{ht}(\beta) = n\}$.

For any $\beta \in \sfq^+_n$, we set $I^\beta = \{\bfi \in I^n \mid \alpha_{i_1} + \dots + \alpha_{i_n} = \beta\}$.
The symmetric group $\sss$ acts on elements of $I^n$ by place permutation.

The \emph{quiver Hecke algebra} $\scrr_\beta$ is the unital associative $\calo$-algebra with generators
\[
\{e(\bfi) \mid \bfi \in I^\beta\} \cup \{x_1, \dots, x_n\} \cup \{\psi_1, \dots, \psi_{n-1}\},
\]
subject to the following relations.

\begin{alignat*}{2}
e(\bfi)e(\bfj)&=\delta_{\bfi, \bfj} e(\bfi); \! \qquad\qquad\qquad\qquad\qquad & x_r e(\bfi) &= e(\bfi) x_r;\\
\sum_{\bfi \in I^\beta} e(\bfi)&=1;  & x_r x_s &= x_s x_r;\\
\psi_r e(\bfi) &= e(s_r\bfi) \psi_r; & \psi_r x_s &= \mathrlap{x_s \psi_r}\hphantom{\smash{\psi_s\psi_r}} \quad \text{if } s\neq r,r+1;\\
x_r \psi_r e(\bfi) &=(\psi_r x_{r+1} - \delta_{i_r,i_{r+1}})e(\bfi); & \psi_r \psi_s &= \psi_s\psi_r \quad \text{if } |r-s|>1;\\
x_{r+1} \psi_r e(i) &=(\psi_r x_r + \delta_{i_r,i_{r+1}})e(\bfi);
\end{alignat*}
\vspace{-4ex}
\begin{align*}
\psi_r^2 e(\bfi) &=
\begin{cases}
(x_r + x_{r+1}^2) e(\bfi) & \text{if } (i_r, i_{r+1}) = (0,1) \text{ or if } (\ell, \ell-1);\\
\\
(x_r^2 + x_{r+1}) e(\bfi) & \text{if } (i_r, i_{r+1}) = (1,0) \text{ or if } (\ell-1, \ell);\\
\\
(x_r + x_{r+1}) e(\bfi) & \text{if } i_{r+1} = i_r \pm 1, \ i_r\neq 0\neq i_{r+1}, \ i_r\neq \ell\neq i_{r+1};\\
\\
0 & \text{if } i_r = i_{r+1};\\
\\
e(\bfi) & \text{otherwise}.
\end{cases}\\
\\
\psi_{r+1} \psi_{r} \psi_{r+1} e(\bfi)&=\begin{cases}
(\psi_{r} \psi_{r+1} \psi_{r} + x_r + x_{r+2}) e(\bfi) & \text{if } (i_r, i_{r+2}, i_{r+1}) = (1, 0,1)\\
& \qquad\qquad\qquad \:\, \text{ or } (\ell-1, \ell, \ell-1);\\
(\psi_{r}\psi_{r+1}\psi_{r} + 1) e(\bfi) & \text{if }i_r = i_{r+2} = i_{r+1} \pm 1, \text{ and } i_{r+1} \neq0,\ell;\\
\\
\psi_{r}\psi_{r+1}\psi_{r} e(\bfi) & \text{otherwise}.\end{cases}
\end{align*}

The quiver Hecke algebra $\scrr_n$ is defined to be $\bigoplus_{\beta\in \sfq^+_n}\scrr_\beta$.
These algebras have cyclotomic quotients, which are our primary interest here.
The \emph{cyclotomic quiver Hecke algebra} $\scrr^\La_\beta$ is the quotient of $\scrr_\beta$ by the additional \emph{cyclotomic relations}
\[
x_1^{\langle \La, \alpha^\vee_{i_1} \rangle} e(\bfi) = 0 \text{ for all } \bfi\in I^\beta.
\]
The \emph{cyclotomic quiver Hecke algebra} $\scrr^\La_n$ is defined to be $\bigoplus_{\beta\in \sfq^+_n}\scrr^\La_\beta$.

The quiver Hecke algebras and their cyclotomic quotients may be ($\bbz$-)graded by
\[
\deg e(\bfi) = 0, \qquad \deg x_r e(\bfi) = (\alpha_{i_r}, \alpha_{i_r}), \qquad \deg \psi_r e(\bfi) = (\alpha_{i_r}, \alpha_{i_{r+1}}),
\]
where $(- ,-)$ is the invariant symmetric bilinear form on $\sfp$.

\begin{rem}
Technically, we have made a choice of certain polynomials in our definition of the quiver Hecke algebras.
See \cite[\S2.1--2.2]{aps} for discussion of these polynomials and the choice we have made.
\end{rem}

\subsection{Multipartitions and tableaux}

A \emph{partition} $\la$ of $n$ is a weakly decreasing sequence of non-negative integers $\la  = (\la_1, \la_2, \dots)$ such that $\sum \la_i = n$.
We write $\varnothing$ for the unique partition of $0$.
For any $l\geq 1$, an \emph{$l$-multipartition} of $n$ is an $l$-tuple $\la = (\la^{(1)}, \dots, \la^{(l)})$.
We denote the set of all $l$-multipartitions of $n$ by $\mptn l n$.
For $\la, \mu \in \mptn l n$, we say that $\la$ \emph{dominates} $\mu$, and write $\la \dom \mu$ or $\mu \domby \la$, if for all $1\leq t\leq l$ and $r\geq0$,
\[
|\la^{(1)}| + \dots + |\la^{(t-1)}| + \sum_{j=1}^r \la^{(t)}_r \geq |\mu^{(1)}| + \dots + |\mu^{(t-1)}| + \sum_{j=1}^r \mu^{(t)}_r.
\]

For $\la\in \mptn l n$, we define the \emph{Young diagram} $[\la]$ to be the set
\[
\{(r,c,t) \in \bbz_{>0} \times \bbz_{>0} \times \{1,\dots, l\} \mid c\leq \la^{(t)}_r\}.
\]
We call elements of $[\la]$ \emph{nodes}.
We draw the Young diagram of a partition using the English convention (where the first coordinate increases down the page and the second coordinate increases from left to right), and of a multipartition as a column vector of Young diagrams for each component.
We say that $A\notin [\la]$ is an \emph{addable node} (for $\la$) if $[\la] \cup A$ is a valid Young diagram of a multipartition.

Let $p$ be the natural projection $p: \bbz \rightarrow \bbz/2\ell\bbz$ if $\ell < \infty$, and $p = \operatorname{id}$ if $\ell = \infty$.
If $\ell = \infty$, we define $f_\ell: \bbz \rightarrow I$ by $k \mapsto |k|$.
If $\ell < \infty$, we define $f_\ell: \bbz/2\ell\bbz \rightarrow I$ by $f_\ell(0) = 0$, $f_\ell(\ell) = \ell$, and $f_\ell(k) = f_\ell(2\ell-k) = k$ for $1\leq k\leq \ell-1$.
Then we define $\overline{\phantom{k}}: f_\ell \circ p:\bbz \rightarrow I$.

Given a \emph{multicharge} $\kappa = (\kappa_1, \dots, \kappa_l)\in \bbz^l$, we define $\La_\kappa \in \sfp^+$ by $\La_\kappa = \La_{\overline{\kappa_1}} + \dots + \La_{\overline{\kappa_l}}$.
The \emph{residue} of a node $A = (r,c,t) \in[\la]$ is $\res A = \overline{\kappa_t + c - r}$.
If $\res A = i$, we call $A$ an $i$-node.

\begin{eg}
Let $\ell = 3$, $\kappa = (1,4)$ and $\la = ((8,3,2),(5,3,1))$.
Then the Young diagram $[\la]$, along with the residue pattern, is depicted below.
\begin{align*}
&\young(12321012,012,10)\\
&\young(21012,321,2)
\end{align*}
\end{eg}

A \emph{$\la$-tableau} is a bijection $\ttt: [\la] \to \{1,\dots, n\}$.
We depict $\ttt$ by filling each node $(r,c,t)$ with $\ttt(r,c,t)$.
We say that a $\la$-tableau is \emph{standard} if in each component, the entries increase along each row and down each column.
We denote by $\mathrm{Std}(\mptn l n)$ the set of all standard tableaux whose shape is an $l$-multipartition of $n$, and by $\mathrm{Std}^2(\mptn l n)$ the subset of $\mathrm{Std}(\mptn l n) \times \mathrm{Std}(\mptn l n)$ consisting of all pairs of standard tableaux of the same shape.

The distinguished tableau $\ttt^\la$ is obtained by filling nodes in order along rows, starting with the first row of $[\la^{(1)}]$ and working down the rows of this component before moving on to successive components.

A \emph{Garnir node} $A = (r,c,t) \in [\la]$ is a node for which $(r+1,c,t) \in [\la]$.
The corresponding \emph{Garnir belt} $\mathbf{B}^A$ is the set of nodes
\[
\{(r,c,t), (r,c+1,t), \dots, (r,\la^{(t)}_r,t)\} \cup \{(r+1,1,t), (r+1,2,t), \dots, (r+1,c,t)\}.
\]
We define the \emph{Garnir tableau} $\ttg^A$ to be the $\la$-tableau which agrees with $\ttt^\la$ outside of $\mathbf{B}^A$, with the entries in $\mathbf{B}^A$ in order from left to right along row $r+1$, and then row $r$.
See \cite[\S1.4]{aps} for examples.

The \emph{residue sequence} of a $\la$-tableau $\ttt$ is the sequence $\bfi^\ttt = (i_1, \dots, i_n)$, where $i_r = \res \ttt^{-1}(r)$.
We define $\bfi^\la = \bfi^{\ttt^\la}$.

We denote by $\shp\ttt$ the \emph{shape} of $\ttt$ -- i.e.~$\ttt$ is a $\shp\ttt$-tableau.
We let $\ttt_{\downarrow_m}$ denote the tableau obtained from $\ttt$ by deleting all entries greater than $m$.
Finally, we define the dominance order on tableaux by $\tts \domby \ttt$ if $\shp{\tts_{\downarrow_m}} \domby \shp{\ttt_{\downarrow_m}}$ for all $1\leq m \leq n$.

For each $w\in \sss$, we fix a preferred reduced expression $w = s_{i_1} \dots s_{i_r}$.
For $\ttt$ a $\la$-tableau, we define $w^\ttt \in \sss$ to be the permutation such that $w^\ttt \ttt^\la = \ttt$, where $\sss$ acts on tableaux by permuting entries.
If $w^\ttt = s_{i_1} \dots s_{i_r}$ is our preferred reduced expression for $w^\ttt$, we define the element $\psi_{w^\ttt} = \psi_{i_1} \dots \psi_{i_r} \in \scrr_n$.

\subsection{Specht modules}\label{section: specht}

We will briefly recall the definition of the (graded) Specht modules from \cite{aps}.
The reader should refer to \cite{aps} for a more thorough treatment of Specht modules, and of the graded module categories of $\scrr^\La_n$.

Fix a multicharge $\kappa \in \bbz^l$ and let $\la\in \mptn l n$.
For each Garnir node $A \in [\la]$ we may define the \emph{Garnir element} $\mathsf g^A \in \scrr_n$.
See \cite[\S3.2]{aps} for the definition of $\mathsf g^A$.

The graded Specht module $\spe\la_\kappa$ is the unital $\scrr_n$-module with generator $z^\la$ of degree $\deg \ttt^\la$ (see \cite[\S1.3]{aps}) subject to the relations

\begin{enumerate}[label=(\roman*)]
\item $e(\bfi) z^\la = z^\la$;
\item $x_r z^\la = 0$ for all $1\leq r \leq n$;
\item $\psi_r z^\la = 0$ whenever $r$ and $r+1$ lie in the same row of $\ttt^\la$;
\item $\mathsf g^A z^\la = 0$ for all Garnir nodes $A \in [\la]$.
\end{enumerate}

For each $\la$-tableau $\ttt$, we define $v^\ttt = \psi_{w^\ttt} z^\la \in \spe\la_\kappa$.

\begin{thmc}{aps}{Theorem~3.12}\label{spechtbasis}
The Specht module $\spe\la_\kappa$ is a graded $\scrr^\La_n$-module and is generated by $\{v^\ttt \mid \ttt \in \std\la\}$ as an $\calo$-module.
\end{thmc}

In type $C_\infty$, \cite[Theorem~3.19]{aps} tells us that the generating set in \cref{spechtbasis} is in fact a (homogeneous) basis, and we conjectured in \cite[Conjecture~5.3]{aps} that the same is true in type $C^{(1)}_\ell$.

For us, it will suffice to note that $\mathsf g^A = \psi_{w^{\ttg^A}}$ for all Specht modules we will consider.
Indeed, we have \cite[Equation~3.3]{aps}:
\[
\mathsf g^A = \psi_{w^{\ttg^A}} + \sum_w a_w \psi_w \quad \text{for some } a_w \in \calo,
\]
where the sum is taken over $\{w\in \sss \mid w < w^{\ttg^A}, \, \bfi^{w \ttt^\la} = \bfi^{\ttg^A}, \, w \ttt^\la \text{ is \emph{row-strict}}\}$.
In every Specht module we will consider in this paper, there is no row-strict $\la$-tableau which dominates $\ttg^A$ and has the same residue sequence as $\ttg^A$, for any Garnir node $A$.
In general, however, $\mathsf g^A$ will also include these terms indexed by more dominant tableaux with the same residue sequence.

The following lemma will be useful to us later.

\begin{lem}\label{gen action on basis}
Let $\la \in \mptn l n$.
Then we have the following actions of the generators of $\scrr_n$ on the $\calo$-generating set for $\spe\la$ in \cref{spechtbasis}.
\begin{enumerate}[label=(\roman*)]
\item Let $\ttt\in \std\la$ and $1\leq r\leq n$.
Then
\[
x_r v^\ttt = \; \sum_{\mathclap{\substack{\tts \in \std\la\\ \bfi^\tts = \bfi^\ttt \\ \tts \doms \ttt}}} \; a_\tts v^\tts \quad \text{for some } a_\tts \in \calo.
\]
\item Let $\ttt\in \std\la$ and $1\leq r< n$.
Then
\[
\psi_r v^\ttt = \; \sum_{\mathclap{\substack{\tts \in \std\la\\ \bfi^\tts = \bfi^{s_r \ttt} \\ \tts \doms \ttt}}} \; a_\tts v^\tts \quad \text{for some } a_\tts \in \calo,
\]
unless $s_r \ttt \in \std\la$ and $s_r w^\ttt$ is a reduced expression of length $\ell(w^\ttt) + 1$.
\end{enumerate}
\end{lem}

\begin{proof}
This is identical to \cite[Lemmas~4.8 and 4.9]{bkw11} and \cite[Lemma~2.14]{fs16}.
\end{proof}

\section{Semisimplicity of $\scrr^\La_n$}

Let $\ell \in \{2,3,\dots\} \cup \{\infty\}$, $\La\in \sfp^+$ be a dominant weight of level $l \in \bbz_{>0}$ and $n\in \bbz_{>1}$ so that we have the cyclotomic quiver Hecke algebra $\scrr^\La_n$.
Let $\kappa\in \bbz^l$ be any multicharge such that $\La = \La_\kappa$.

For $i\in I$ and $k\in \bbz_{>0}$, we set $\alpha^\vee_{i,k} = \alpha^\vee_i + \alpha^\vee_{i+1} + \dots + \alpha^\vee_{i+k-1}$, where the indices are taken modulo $\ell+1$.

The following two conditions will be key in our semisimplicity arguments, and we will refer back to them frequently.

\begin{enumerate}[label=(SS\arabic*)] 
\item\label{ss1} For all $i \in I$, $\langle\La,\alpha^\vee_{i,n}\rangle \leq 1$.
\item\label{ss2} For all $1\leq j\leq l$, $\tfrac{n-1}{2} \leq \overline{\kappa_j} \leq \ell - \tfrac{n-1}{2}$.
\end{enumerate}

\begin{rem}
The following observations are the driving force for this paper, and will be used frequently.
\begin{enumerate}
\item Suppose that \ref{ss1} holds, and let $\la \in \mptn l m$ for some $0\leq m < n$.
Then for any $i\in I$, $\mu$ has at most one component with addable $i$-nodes.
Informally, we may think of \ref{ss1} as ensuring that for any $\la \in \mptn l n$, nodes in distinct components of $[\la]$ must have distinct residues.
\item Suppose that \ref{ss2} holds, and let $\la \in \mptn l n$.
For a given residue $i\in I$, there is either only one possible diagonal of residue $i$ which may appear in the Young diagram of some partition, or there are two diagonals which, in any multipartition, may contain at most a single node each (in which case both nodes lie in the same row or the same column of the multipartition, and the residue is either $1$ or $\ell-1$).
\end{enumerate}
\end{rem}

\subsection{The semisimple case}

First, we will handle the case when $\scrr^\La_n$ is semisimple.
This subsection mirrors the corresponding type $A$ arguments of \cite[\S2.4]{m14surv}, which we have adapted to fit the type $C$ case.

\begin{lem}
\label{residue}
Suppose that conditions \ref{ss1} and \ref{ss2} hold, and let $\ttt,\tts \in \std{\mptn l n}$.
Then $\ttt = \tts$ if and only if $\bfi^\ttt = \bfi^\tts$.
\end{lem}

\begin{proof}
If $i\in I$ and $\la \in \mptn l m$ for some $0\leq m < n$.
Then by the above remark, $[\la]$ has at most one addable $i$-node, and the result follows by induction on $n$.
%If $i\in I$ and $\la \in \mptn l m$ for some $0\leq m < n$, then $\mu$ has at most one component with addable $i$-nodes by \ref{ss1}.
%Furthermore, \ref{ss2} ensures that any component contains either one diagonal of any given residue, or two diagonals containing a single node each (in which case both nodes lie in the same row or the same column of the multipartition, and the residue is either $1$ or $\ell-1$).
%Thus $[\la]$ has at most one addable $i$-node, and the result follows by induction on $n$.
\end{proof}

Let $I^n_\La = \{\bfi^\ttt \mid \ttt \in \std{\mptn l n}\}$.

\begin{cor}
\label{leaveresidue}
Suppose that conditions \ref{ss1} and \ref{ss2} hold, and let $\bfi \in I^n_\La$ such that $i_{r+1} = i_r \pm 1$.
Then $s_r \bfi \notin I^n_\La$.
\end{cor}

\begin{proof}
By the above remark, if $\bfi = \bfi^\ttt$ for some $\ttt\in \std{\mptn l n}$, and $i_{r+1} = i_r \pm 1$, then $r$ and $r+1$ must lie in adjacent diagonals of $\ttt$.
In particular, they must lie in either the same row or the same column of $\ttt$.
Given that the number of residues equal to $i_r$ and $i_{r+1}$ in $(i_1,\dots,i_{r-1})$ is unchanged, we deduce that $r$ and $r+1$ must occupy the same two pair of nodes in $\ttt$ as in any standard tableau with residue sequence $s_r \bfi$.
But this is a contradiction, as such a tableau cannot be standard.
\end{proof}

Recall that we have fixed a multicharge $\kappa\in \bbz^l$ such that $\La = \La_\kappa$.

\begin{thm}\label{irred}
Suppose that $\calo = \bbf$ is a field, and that conditions \ref{ss1} and \ref{ss2} hold.
Then for each $\la\in\mptn l n$ there is an irreducible graded $\scrr^\La_n$-module $\rspe\la_\kappa$ with homogeneous basis $\{v^\ttt\mid \ttt\in \std\la\}$ such that $\deg v^\ttt = 0$ for all $\ttt\in\std\la$, and the $\scrr^\La_n$-action is given by
\[
e(\bfi)v^\ttt = \delta_{\bfi,\bfi^\ttt} v^\ttt, \quad x_r v^\ttt = 0, \quad \psi_r v^\ttt = v^{s_r\ttt},
\]
where we set $v^{s_r\ttt} = 0$ if $s_r\ttt$ is not standard.
\end{thm}

\begin{proof}
We first check that the relations above really define an $\scrr^\La_n$-module.
Almost all the defining relations for $\scrr^\La_n$ are trivially satisfied, thanks to \cref{residue,leaveresidue}.
We must check that the $\psi$ generators satisfy the braid relations and the quadratic relations when acting on basis elements.
Let $\la \in \mptn l n$, $\ttt \in \std\la$ and set $\bfi = \bfi^\ttt = (i_1,\dots, i_n)$.

For the braid relations, \ref{ss1} and \ref{ss2} ensure that we never have $i_r = i_{r+2} = i_{r+1} \pm 1$ with $i_{r+1} \neq0,\ell$.
To see this, we again invoke our remark made after introducing conditions \ref{ss1} and \ref{ss2}.
Since we can only have a single diagonal of any residue besides $1$ and $\ell$, it is not possible for the (arbitrarily chosen) standard tableau $\ttt$ to have consecutive residues $i, i\pm1, i$, except for $(1,0,1)$ and $(\ell-1, \ell, \ell-1)$.
Finally, if $(i_r, i_{r+1}, i_{r+2}) = (1,0,1)$ or $(\ell-1, \ell, \ell-1)$, then we have $\psi_{r+1} \psi_r \psi_{r+1} v^\ttt = \psi_r \psi_{r+1} \psi_r v^\ttt = 0$.

%For the braid relations, \ref{ss1} and \ref{ss2} ensure that we never have $i_r = i_{r+2} = i_{r+1} \pm 1$ with $i_{r+1} \neq0,\ell$, and that when $(i_r, i_{r+1}, i_{r+2}) = (1,0,1)$ or $(\ell-1, \ell, \ell-1)$, $\psi_{r+1} \psi_r \psi_{r+1} v^\ttt = \psi_r \psi_{r+1} \psi_r v^\ttt = 0$.

Since $i_{r+1} \neq i_r$ for any $r$ and $i_{r+1} = i_r \pm 1$ if and only if $r$ and $r+1$ are in the same row or column of $\ttt$, it follows from \cref{leaveresidue} that $\psi_r^2 v^\ttt = 0$ when $i_{r+1} = i_r \pm 1$.

These residue conditions also tell us that $\deg \psi_r e(\bfi) = 0$ whenever $s_r \ttt \in \std\la$ (and if $s_r \ttt \notin \std\la$, $\psi_r e(\bfi) = 0$ by \cref{leaveresidue}).
Thus setting $\deg v^\ttt = 0$ gives a grading on $\rspe\la_\kappa$.

Finally, we show that $\rspe\la_\kappa$ is irreducible.
If $\tts,\ttt \in \std\la$, then $\tts = w^\tts \ttt^\la = w^\tts (w^\ttt)^{-1} \ttt$.
So $v^\tts = \psi_{w^\tts} \psi_{(w^\ttt)^{-1}} v^\ttt$.
Take a non-zero element $v = \sum_{\ttt\in\std\la} a_\ttt v^\ttt \in \rspe\la_\kappa$.
If $a_\ttt \neq 0$ then, by \cref{residue}, $v^\ttt = \tfrac{1}{a_\ttt} e(\bfi^\ttt) v$, and therefore for any $\tts \in \std\la$, $v^\tts \in \scrr^\La_n v$.
It follows that $\rspe\la_\kappa$ is irreducible.
\end{proof}

\begin{rem}
The modules $\rspe\la_\kappa$ are easily seen to be isomorphic to the Specht modules $\mathcal{S}^\la_\kappa$ constructed in \cref{section: specht}, providing evidence for the importance of the Specht modules constructed in \cite{aps}.
Indeed, as remarked after \cref{spechtbasis}, we know that $\mathsf g^A = \psi_{w^{\ttg^A}}$, and this is sufficient to prove that $\mathcal{S}^\la_\kappa$ has a basis indexed by standard tableaux (the elements constructed in \cite[Theorem 3.12 and Corollary 3.13]{aps}), showing that the dimensions match.
By the definition of $\mathcal{S}^\la_\kappa$, the cyclic generator $z^\la$ satisfies the same relations as the element $v^{\ttt^\la}$ constructed in \cref{irred}, so that we have an isomorphism $\rspe\la_\kappa \rightarrow \mathcal{S}^\la_\kappa$ determined by $v^{\ttt^\la} \mapsto z^\la$.
\end{rem}

If $\bfi = (i_1,\dots, i_n) \in I^n$, we set $\bfi_{\downarrow_r} = (i_1, \dots, i_r)$.

\begin{lem}
\label{neighbourres}
Suppose that conditions \ref{ss1} and \ref{ss2} hold, and let $\bfi \in I^n$.
Then $\bfi \in I^n_\La$ if and only if $\bfi$ satisfies the following three conditions.
\begin{enumerate}[label=(\roman*)]
\item $\langle \La,\alpha^\vee_{i_1} \rangle \neq 0$.
\item If $1< r \leq n$ and $\langle \La, \alpha^\vee_{i_r} \rangle = 0$, then $\{\overline{i_r - 1}, \overline{i_r + 1}\} \cap \{i_1, \dots, i_{r-1}\} \neq \emptyset$.
\item Let $1\leq s < r \leq n$.
If $i_r = i_s\neq 1,\ell-1$, then $\{\overline{i_r - 1}, \overline{i_r + 1}\} \subseteq \{i_{s+1}, \dots, i_{r-1}\}$.
If $i_r = i_s = 1$, then $0 \in \{i_{s+1}, \dots, i_{r-1}\}$.
If $i_r = i_s = \ell-1$, then $\ell \in \{i_{s+1}, \dots, i_{r-1}\}$.
\end{enumerate}
\end{lem}

\begin{proof}
Let $\ttt \in \std{\mptn l n}$ with $\bfi^\ttt = \bfi$.
We prove by induction on r that $\bfi_{\downarrow_r} \in I^r_\La$ satisfies all three conditions as claimed.
By definition, $i_1 = \overline{\kappa_j}$ for some $j$, so \textit{(i)} holds.
By induction, we assume that $\bfi_{\downarrow_{r-1}}$ satisfies \textit{(i)--(iii)}.
If $\langle \La, \alpha^\vee_{i_r}\rangle = 0$, then $r$ is not in the $(1,1)$ node of any component of $\ttt$, so $\ttt$ has an entry directly above or to the left of $r$, so \textit{(ii)} holds.
Now suppose that $i_r = i_s\neq 1,\ell-1$ are as in the first part of \textit{(iii)}.
Condition \ref{ss1} ensures that residues in different components are distinct, so that $r$ and $s$ must be in the same component of $\ttt$.
Condition \ref{ss2} ensures that $r$ and $s$ are on the same diagonal, so that $r$ is not in the first row or first column of the component, so \textit{(iii)} holds.
Finally, suppose that $i_r = i_s = 1$ or $\ell-1$.
Then we have $r$ and $s$ both appearing in the first row or both appearing in the first column of $\ttt$, so that $\{i_{s+1}, \dots, i_{r-1}\}$ contains $0$ if $i_r = 1$, or $\ell$ if $i_r = \ell-1$, proving the second and third statements in \textit{(iii)}.

Conversely, suppose that $\bfi\in I^n$ satisfies conditions \textit{(i)--(iii)}.
We show by induction on $r$ that $\bfi_{\downarrow_r} \in I^r_\La$ for $1\leq r \leq n$.
If $r=1$, \textit{(i)} implies that $\bfi_{\downarrow_r} \in I^r_\La$.
So suppose by the induction hypothesis that for some $1< r <n$, $\bfi_{\downarrow_r} = \bfi^\tts$ for some $\tts \in \std{\mptn l r}$.
Let $\la = \shp\tts$.
From the proof of \cref{residue}, we know that for any $i\in I$, $[\la]$ has at most one addable $i$-node.

If $\langle \La, \alpha^\vee_{i_{r+1}}\rangle = 0$, then by \textit{(ii)}, $\la$ contains either an $(\overline{i_{r+1}-1})$-node or $(\overline{i_{r+1}+1})$-node (or both).
Thus either there is an addable $i_{r+1}$-node in the first row or first column of the corresponding component of $\la$, or else there is some $1\leq s < r+1$ such that $i_s = i_{r+1}$.
By \ref{ss2}, if there is no addable $i_{r+1}$-node in the first row or column, then $1< i_{r+1} < \ell-1$ in this case, and condition \textit{(iii)} tells us that there is again an addable $i_{r+1}$-node.

If $\langle \La, \alpha^\vee_{i_{r+1}}\rangle = 1$, then either the $(1,1)$ node of some component of $[\la]$ is an addable $i_{r+1}$-node, or $[\la]$ already contains a $(1,1)$ node which has residue $i_{r+1}$.
In the latter case, \textit{(iii)} implies that $[\la]$ has an addable $i_{r+1}$-node.

Thus we know that $[\la]$ has precisely one addable $i_{r+1}$-node, which we shall denote by $A$.
Then $\bfi_{\downarrow_{r+1}} = \bfi^\ttt$ where $\ttt$ is the unique standard tableau satisfying $\ttt_{\downarrow_r} = \tts$ and $\ttt (A) = r+1$.
Hence $\bfi_{\downarrow_{r+1}} \in I^{r+1}_\La$ and the proof is complete.
\end{proof}

The following lemma follows easily from the rank formula for $e(\bfi) \scrr^\La_n e(\bfi)$ given in~\cite[Theorem~2.5]{aps}, and does not require that \ref{ss1} and \ref{ss2} are satisfied.

\begin{lem}\label{idemps0}
If $\bfi \in I^n \setminus I^n_\La$, then $e(\bfi) = 0$ in $\scrr^\La_n$.
\end{lem}

\begin{lem}
\label{x=0}
Let $1\leq m < n$ and suppose that \ref{ss1} and \ref{ss2} hold if $n$ is replaced with $m$.
Then $x_1 = \dots = x_m = 0$.
\end{lem}

\begin{proof}
Using the defining relations of $\scrr^\La_n$, we will prove by induction on $r$ that $x_r e(\bfi) = 0$ for all $\bfi \in I^n_\La$ and $1 \leq r\leq m$, from which the result will follow by \cref{idemps0}.

%We prove by induction on $r$ that $x_r = 0$ for $1 \leq r\leq m$.
When $r=1$, the result follows immediately from the cyclotomic relations.
So we will assume that $x_1 = \dots = x_{r-1} = 0$, and show that $x_r e(\bfi) = 0$ whenever $\bfi_{\downarrow_r} \in I^r_\La$.
%So we will assume that $x_1 = \dots = x_{r-1} = 0$.
%By \cref{idemps0}, it suffices to show that $x_r e(\bfi) = 0$ whenever $\bfi_{\downarrow_r} \in I^r_\La$.

If $i_{r-1} = i_r \pm 1$ and neither $i_{r-1}$ nor $i_r$ are $0$ or $\ell$, then by induction we have
\[
x_r e(\bfi) = (x_r + x_{r-1}) e(\bfi) = \psi_{r-1}^2 e(\bfi) = \psi_{r-1} e(s_{r-1}\bfi) \psi_{r-1} = 0,
\]
where the last equality follows from \cref{leaveresidue}.
Similarly, if $(i_{r-1}, i_r) = (1,0)$ or $(\ell-1, \ell)$, then
\[
x_r e(\bfi) = (x_r + x_{r-1}^2) e(\bfi) = \psi_{r-1}^2 e(\bfi) = \psi_{r-1} e(s_{r-1}\bfi) \psi_{r-1} = 0.
\]
If $(i_{r-1}, i_r) = (0,1)$ or $(\ell, \ell-1)$, then by \ref{ss2} and \cref{residue}, $\bfi$ is the residue sequence of some standard tableau $\ttt$ of shape $\la$, and $[\la]$ has exactly one other $1$-node (resp.~$(\ell-1)$-node) besides $\ttt^{-1}(r)$, and $\ttt^{-1}(r-1)$ is the only $0$-node (resp.~$\ell$-node) of $[\la]$.
Moreover, the other $1$-node (resp.~$(\ell-1)$-node) is $\ttt^{-1}(u)$ for some $1\leq u < r$, and $\ttt^{-1}(v) \neq 2$ (resp.~$\ell-2$) for any $u < v < r$.
Thus we have
\begin{align*}
e(\bfi) &= \psi_u^2 e(\bfi) = \psi_u e(s_u\bfi) \psi_u = \psi_u \psi_{u+1} e(s_{u+1} s_u \bfi) \psi_{u+1} \psi_u\\
&= \dots = \psi_u \psi_{u+1} \dots \psi_{r-3} e(s_{r-3}\dots s_u\bfi) \psi_{r-3} \dots \psi_u,
\end{align*}
so that $((s_{r-3}\dots s_u\bfi)_{r-2}, (s_{r-3}\dots s_u\bfi)_{r-1}, (s_{r-3}\dots s_u\bfi)_r) = (1,0,1)$ or $(\ell-1, \ell, \ell-1)$ and
\begin{align*}
x_r e(\bfi) &= \psi_u \psi_{u+1} \dots \psi_{r-3} x_r e(s_{r-3}\dots s_u\bfi) \psi_{r-3} \dots \psi_u\\
&= \psi_u \psi_{u+1} \dots \psi_{r-3} (\psi_{r-1} \psi_{r-2} \psi_{r-1} - \psi_{r-2} \psi_{r-1} \psi_{r-2} - x_{r-2}) e(s_{r-3}\dots s_u\bfi) \psi_{r-3} \dots \psi_u\\
&= 0
\end{align*}
by the induction hypothesis and the fact that $s_{r-2} s_{r-3}\dots s_u\bfi, s_{r-1}s _{r-3}\dots s_u\bfi \notin I^n_\La$ by \cref{leaveresidue}.

Finally, if $i_{r-1} \neq i_r \pm 1$, then since we know that $i_{r-1} \neq i_r$ by \cref{neighbourres},
\[
x_r e(\bfi) = x_r \psi_{r-1}^2 e(\bfi) = x_r \psi_{r-1} e(s_{r-1} \bfi) \psi_{r-1} = 0,
\]
which completes the proof.
\end{proof}

\begin{defn}
Let $(\tts,\ttt) \in \mathrm{Std}^2(\mptn l n)$.
Then we define the element $e_{\tts\ttt}\in \scrr^\La_n$ to be $e_{\tts\ttt} = \psi_{(w^\tts)^{-1}} e(\bfi^\la)\psi_{w^\ttt}$.
\end{defn}

By \cref{residue}, the elements $e_{\tts\ttt}$ do not depend on the choice of reduced expression.

\begin{thm}
\label{Ccellular}
Suppose that conditions \ref{ss1} and \ref{ss2} hold.
Then $\scrr^\La_n$ is a graded cellular algebra with graded cellular basis
\[
\scrb = \{e_{\tts\ttt} \mid (\tts,\ttt) \in \mathrm{Std}^2(\mptn l n)\}
\]
with $\deg e_{\tts\ttt} = 0$ for all $\tts,\ttt$.
\end{thm}

\begin{proof}
By \cref{neighbourres}, if $\bfi \in I^n_\La$, $\bfi$ cannot contain a subsequence of the form $(i, i\pm1, i)$ for any $i\in I$, except possibly $(1,0,1)$ or $(\ell-1, \ell, \ell-1)$.
\cref{idemps0,x=0} imply that (even in the degenerate cases above) the $\psi$ generators satisfy the braid relations for $\sss$.
Therefore $\scrr^\La_n$ is spanned by the elements $\{\psi_v e(\bfi) \psi_w \mid v,w \in \sss, \; \bfi \in I^n_\La\}$.
Since $\psi_v e(\bfi) \psi_w = e(v \bfi) \psi_v e(\bfi) \psi_w = 0$ if $v\bfi \notin I^n_\La$, $\scrr^\La_n$ is in fact spanned by the elements  of $\scrb$.
It follows from the rank formula~\cite[Theorem~2.5]{aps} that $\scrb$ is a basis for $\scrr^\La_n$.

The orthogonality relations on the idempotents $e(\bfi)$ imply that $e_{\tts\ttt} e_{\ttu\ttv} = \delta_{\ttt,\ttu} e_{\tts\ttv}$, so that $\scrb$ is in fact a basis of matrix units, and
\[
\scrr^\La_n = \bigoplus_{\la \in \mptn l n} \operatorname{Mat}_{\dim \rspe\la_\kappa}(\calo).
\]
It follows that this basis is a cellular basis.
As in the proof of \cref{irred}, we have that $\deg \psi_r e(\bfi) = 0$ for all $1\leq r < n$ and $\bfi\in I^n_\La$, so all elements of $\scrb$ are homogeneous of degree $0$.
\end{proof}

In the proof of the above theorem, we showed that if conditions \ref{ss1} and \ref{ss2} hold, $\scrr^\La_n$ is a direct sum of matrix algebras.
We obtain the main result of this subsection as a corollary of this fact.

\begin{thm}\label{ss}
Suppose that $\calo = \bbf$ is a field and that conditions \ref{ss1} and \ref{ss2} hold.
Then $\scrr^\La_n$ is semisimple.
\end{thm}

\subsection{The non-semisimple case}

In this section, we will assume throughout that $\calo = \bbf$ is a field and prove the following converse to \cref{ss}.

\begin{thm}
Suppose that $\calo = \bbf$ is a field, and that at least one of the conditions \ref{ss1} and \ref{ss2} fails.
Then $\scrr^\La_n$ is not semisimple.
\end{thm}

We break the proof into several lemmas.
First we will look at the case where \ref{ss2} fails.
We begin with separate treatment of the case where $\overline{\kappa_j} = 0$ or $\ell$ for some $1\leq j\leq l$.

\begin{lem}\label{Lambda0notss}
Suppose $\overline{\kappa_j} = 0$ or $\ell$ for some $1\leq j\leq l$.
If $n > 1$,
then $\scrr^\La_n$ is not semisimple.
\end{lem}

\begin{proof}
For any $n>1$ we construct an explicit two-dimensional uniserial $\scrr^\La_n$-module.
Let $\la\in \mptn l n$ be the multipartition such that every component is empty except for component $j$, with $\la^{(j)} = (n)$, and let $\bfi = \bfi^\la$.

Define $M$ to be the $\scrr^\La_n$-module with generators $u,v$ subject to the following relations.

\begin{align*}
e(\bfi) u &= u,\\
e(\bfi) v &= v,\\
\psi_r u &= \psi_r v = 0 \text{ for all } r,\\
x_r u &= 0 \text{ for all } r,\\
x_r v &= 0 \text{ if } r\equiv1 \mod \ell,\\
x_{2k\ell +r} v &= (-1)^r u \text{ for all } k \text{ and all } 2\leq r\leq \ell,\\
x_{2k\ell +r} v &= (-1)^{r+1} u \text{ for all } k \text{ and all } \ell+2 \leq r\leq 2\ell.
\end{align*}

Then $M$ is a two-dimensional vector space over $\bbf$, and we must show it is an $\scrr^\La_n$-module, whence the result follows since $u$ generates a proper submodule of $M$, while $v$ generates the whole of $M$.
So we must check that the defining relations of $\scrr^\La_n$ hold when acting on $M$.
For most of the relations, the result is trivial -- if $\psi$ generators appear in every term or for the idempotent relations, or the product of two $x$ generators.
By definition of $\bfi$, there are no error terms in the relations pushing $x$ generators past $\psi$ generators, so these are also trivial.
This leaves the quadratic and braid relations.

First, we deal with the quadratic relations.
If $r \equiv 1 \mod \ell$, then $(i_r, i_{r+1}) = (0,1)$ or $(\ell, \ell-1)$, so that $\psi_r^2 e(\bfi) = (x_r + x_{r+1}^2) e(\bfi)$.
In both cases, $\psi_r$, $x_r$ and $x_{r+1}^2$ each kill both $u$ and $v$, so the relation holds.
If $(i_r, i_{r+1}) = (1,0)$ or $(\ell-1, \ell)$, then $\psi_r^2 e(\bfi) = (x_r^2 + x_{r+1}) e(\bfi)$, and again each of $\psi_r$, $x_r^2$ and $x_{r+1}$ kills both $u$ and $v$, so the relation holds.
Finally, suppose that $i_r, i_{r+1} \neq 0$ or $\ell$.
Then $\psi_r^2 e(\bfi) = (x_r + x_{r+1}) e(\bfi)$, with the left-hand side killing $u$ and $v$, $x_r$ and $x_{r+1}$ each killing $u$, and $x_r v = - x_{r+1} v$, so that this relation always holds.

Next, we check the braid relations.
Since $\psi_r u = \psi_r v = 0$ for all $r$, we only have to worry about the braid relations which yield error terms.
With our chosen $\bfi$, this only happens for the relations $(\psi_{r+1} \psi_r \psi_{r+1} - \psi_r \psi_{r+1} \psi_r) e(\bfi) = (x_r + x_{r+2}) e(\bfi)$ for $r \equiv 0 \mod \ell$.
Now we have that $x_r v = - x_{r+2}v$ by the final two defining relations for $M$.
\end{proof}

Next, we will handle the case where \ref{ss2} fails and $\overline{\kappa_j} \neq 0, \ell$ for any $1\leq j\leq l$.
Recall that we have fixed a multicharge $\kappa = (\kappa_1, \dots, \kappa_l)\in \bbz^l$ such that $\La = \La_\kappa$.
Define $\overline{\kappa} = (\overline{\kappa_1}, \dots, \overline{\kappa_l})\in I^l$ and $\hat{\kappa} = (2\ell - \overline{\kappa_1}, \dots, 2\ell - \overline{\kappa_l})\in I^l$.

\begin{lem}\label{ss2fail}
Suppose that $\overline{\kappa_j}\neq 0,\ell$ for all $1\leq j \leq l$, and \ref{ss2} fails.
Then $\scrr^\La_n$ is not semisimple.
\end{lem}

\begin{proof}
We fix $1\leq j \leq l$ such that either $\tfrac{n-1}{2} > \overline{\kappa_j}$ or $\ell - \tfrac{n-1}{2} < \overline{\kappa_j}$.
Set $\mu\in \mptn l n$ to be the multipartition such that every component is empty except component $j$, with $\mu^{(j)} = (1^n)$.

If $\tfrac{n-1}{2} > \overline{\kappa_j}$,
we set $\la\in \mptn l n$ to be the multipartition such that every component is empty except component $j$, with $\la^{(j)} = (n - 2\overline{\kappa_j}, 1^{2\overline{\kappa_j}})$.
We will show that for $\ttt$ the least dominant standard $\la$-tableau, the homogeneous basis element $v^\ttt = \psi_{w^\ttt} z^\la$ generates a one-dimensional submodule of $\spe\la_{\overline{\kappa}}$ isomorphic to $\spe\mu_{\overline{\kappa}}$.

If $\ell - \tfrac{n-1}{2} < \overline{\kappa_j}$,
we may instead set $\la\in \mptn l n$ to be the multipartition such that every component is empty except $\la^{(j)} = (2(\ell - \overline{\kappa_j}),1^{n - 2(\ell - \overline{\kappa_j}}))$.
A similar argument shows that for $\ttt$ the least dominant standard $\la$-tableau, $v^\ttt$ generates a one-dimensional submodule of $\spe\la_{\hat{\kappa}}$ isomorphic to $\spe\mu_{\hat{\kappa}}$, so we will focus on the former case, leaving the latter as an exercise.

We have that $e(\bfi^\ttt) v^\ttt = v^\ttt$, where
$\bfi^\ttt = (\overline{\kappa_j}, \overline{\kappa_j}-1,\dots, 1, 0, 1, \dots, \overline{\kappa_j}, \overline{\kappa_j}+1, \dots)$.
We will show that all $x$ and $\psi$ generators of $\scrr^\La_n$ annihilate $v^\ttt$.
First, let $1\leq r\leq n$.
Then by \cref{gen action on basis}(i),
\[
x_r v^\ttt = \; \sum_{\mathclap{\substack{\tts\in \std\la\\ \bfi^\tts = \bfi^\ttt\\ \tts \doms \ttt}}} \; a_\tts v^\tts.
\]
However, it is clear that $\ttt$ is the only standard $\la$-tableau with residue sequence $\bfi$, so that $x_r v^\ttt = 0$.
Now suppose that $1\leq r < n$.
Then since $\ttt$ is the least dominant standard $\la$-tableau, we see by \cref{gen action on basis}(ii) that
\[
\psi_r v^\ttt = \; \sum_{\mathclap{\substack{\tts\in \std\la\\ \bfi^\tts = s_r \bfi^\ttt\\ \tts \doms \ttt}}} \; a_\tts v^\tts.
\]
However, there is no standard $\la$-tableau with residue sequence $s_r \bfi^\ttt$, so that $\psi_r v^\ttt = 0$.

To see that $\spe\mu_{\overline{\kappa}}$ has no complement in $\spe\la_{\overline{\kappa}}$ (i.e.~is not a direct summand), it suffices to note that the residue sequence of the unique standard $\mu$-tableau is different to the residue sequence $\bfi^\la$ of the initial $\la$-tableau $\ttt^\la$, so that there is no non-zero homomorphism $\spe\la_{\overline{\kappa}} \rightarrow \spe\mu_{\overline{\kappa}}$.
\end{proof}

\begin{rem}
Our choice of multicharge defining the Specht modules in \cref{ss2fail} ensures (since $\overline{\kappa_j}\neq \ell$) that $\res(1,2,j) = \overline{\kappa_j} + 1$.
Similarly, in the case left as an exercise, $\res(2,1,j) = \overline{\kappa_j} + 1$.
Thanks to the symmetry in the type $C$ residue pattern, this suffices to prove that $\scrr^\La_n$ is not semisimple.
A different choice of multicharge $\kappa'$ satisfying $\Lambda = \Lambda_{\overline{\kappa'}}$ would also do the trick, but would need a slightly different choice of multipartition $\la$.
\end{rem}

We now turn our attention to the case where condition \ref{ss1} fails.

\begin{lem}\label{kapparepeat}
Suppose that condition \ref{ss2} holds, but $\overline{\kappa_j} = \overline{\kappa_{j'}}$ for some $1\leq j\neq j' \leq l$.
Then $\scrr^\La_n$ is not semisimple.
\end{lem}

\begin{proof}
The proof is similar to the proof of \cref{Lambda0notss}.
We let $\la\in \mptn l n$ be the multipartition such that every component is empty except for component $j$, with $\la^{(j)} = (n)$, and let $\bfi = \bfi^\la$.

Define $M$ to be the $\scrr^\La_n$-module with generators $u,v$ subject to the following relations.

\begin{align*}
e(\bfi) u &= u,\\
e(\bfi) v &= v,\\
\psi_r u &= \psi_r v = 0 \text{ for all } r,\\
x_r u &= 0 \text{ for all } r,\\
x_{\ell - \overline{\kappa_j} + 1} v &= 0,\\
(-1)^{r+1} x_r v &= u \text{ for all } 1\leq r< \ell - \overline{\kappa_j} + 1,\\
(-1)^{r} x_r v &= u \text{ for all } \ell - \overline{\kappa_j} + 1 < r\leq n.
\end{align*}

Then $M$ is a two-dimensional vector space over $\bbf$, and we proceed to show that it is an $\scrr^\La_n$-module.
For most of the relations, the result is trivial, so we check the quadratic and braid relations.
We note that since \ref{ss2} holds, $\bfi$ is a prefix of $(\overline{\kappa_j}, \overline{\kappa_j}+1, \dots, \ell-1, \ell, \ell-1, \dots, \overline{\kappa_j})$, so that there is only a single non-trivial braid relation to check, corresponding to $(i_{\ell - \overline{\kappa_j}}, i_{\ell - \overline{\kappa_j} + 1}, i_{\ell - \overline{\kappa_j} + 2}) = (\ell-1, \ell, \ell-1)$.

First, we deal with the quadratic relations.
For $1\leq r< \ell - \overline{\kappa_j}$, we have $\psi_r^2 e(\bfi) = (x_r + x_{r+1}) e(\bfi)$, and both sides kill $u$ and $v$.
Next, $\psi_{\ell - \overline{\kappa_j}}^2 e(\bfi) = (x_{\ell - \overline{\kappa_j}}^2 + x_{\ell - \overline{\kappa_j} + 1}) e(\bfi)$, and both sides again kill $u$ and $v$.
Similarly, both sides of $\psi_{\ell - \overline{\kappa_j} + 1}^2 e(\bfi) = (x_{\ell - \overline{\kappa_j} + 1} + x_{\ell - \overline{\kappa_j} + 2}^2) e(\bfi)$ kill $u$ and $v$.
Finally, for $\ell - \overline{\kappa_j} + 1 < r < n$, we have $\psi_r^2 e(\bfi) = (x_r + x_{r+1}) e(\bfi)$ and both sides kill $u$ and $v$.

Finally, we check the non-trivial braid relation, which is only present if $n > \ell - \overline{\kappa_j} + 2$.
We have
\[
(\psi_{\ell - \overline{\kappa_j} + 1} \psi_{\ell - \overline{\kappa_j}} \psi_{\ell - \overline{\kappa_j} + 1} - \psi_{\ell - \overline{\kappa_j}} \psi_{\ell - \overline{\kappa_j} + 1} \psi_{\ell - \overline{\kappa_j}}) e(\bfi) =  (x_{\ell - \overline{\kappa_j}} + x_{\ell - \overline{\kappa_j} + 2}) e(\bfi).
\]
Both sides of the above equation kill $u$ and $v$, which completes our proof, as $M$ is uniserial.
\end{proof}

\begin{lem}\label{ss1fail}
Suppose that condition \ref{ss2} holds, $\overline{\kappa_j}$ are distinct, but for some $i\in I$, $\langle \La, \alpha^\vee_{i,n} \rangle >1$.
Then $\scrr^\La_n$ is not semisimple.
\end{lem}

\begin{proof}
In spirit, the proof is the same as that of \cref{ss2fail}.
Since we have assumed that condition \ref{ss2} holds, $\ell\geq n-1$ and we may assume that $i = \overline{\kappa_j}$ for some $1\leq j \leq l$, and for some $1\leq j' \leq l$ and $1\leq k \leq \ell - i - \tfrac{n-1}{2}$, $\overline{\kappa_{j'}} = i + k$.

We consider two cases -- either $j < j'$ or $j > j'$.
As in the proof of \cref{ss2fail}, we will in each case define a multipartition $\la\in \mptn l n$ and let $\ttt$ denote the least dominant standard $\la$-tableau, and will show that $v^\ttt = \psi_{w^\ttt} z^\la = \psi_1 \psi_2 \dots \psi_{n-1} z^\la$ generates a one-dimensional submodule of $\spe\la_{\overline{\kappa}}$.

First suppose that $j < j'$.
Then we define $\la \in \mptn l n$ to be the multipartition with all components empty except components $j$ and $j'$, with $\la^{(j)} = (1^{n - k})$ and $\la^{(j')} = (1^{k})$.
Note that the standard $\la$-tableaux are uniquely determined by their residue sequences, by \cref{residue}.
Now it follows from \cref{gen action on basis} that all $x$ and $\psi$ generators except possibly $\psi_k$ annihilate $v^\ttt$.
We note that $\psi_{w^\ttt}$ is fully commutative, and has an expression starting with $\psi_k$.
Let $\tts$ denote the tableau $s_k \ttt$, so that $v^\ttt = \psi_k \psi_{w^\tts} z^\la$.
Then
\[
\psi_k v^\ttt = \psi_k^2 \psi_{w^\tts} z^\la = (x_k + x_{k + 1}) \psi_{w^\tts} z^\la = 0,
\]
where the last equality follows from \cref{gen action on basis}(i).
We have proved that $v^\ttt$ generates a one-dimensional submodule of $\spe\la_{\overline{\kappa}}$.
As in the proof of \cref{ss2fail}, examining residues yields that this module is not a direct summand of $\spe\la_{\overline{\kappa}}$.

Finally, if $k>0$ and $j > j'$, we define $\la \in \mptn l n$ to be the multipartition with all components empty except components $j$ and $j'$, with $\la^{(j)} = (k)$ and $\la^{(j')} = (n-k)$.
This case is almost identical to the other, and we leave the details to the reader.
\end{proof}

Combining \cref{ss,Lambda0notss,ss2fail,kapparepeat,ss1fail}, we have proved our main theorem, \cref{main}.

\bibliographystyle{amsalpha}  %(other styles: ieeetr, plain, amsplain, siam, unsrt, abbrv, apalike, alpha, amsalpha)
\addcontentsline{toc}{section}{\refname}
\bibliography{master}

\end{document}